\documentclass[11pt,letterpaper]{amsart}
\usepackage{graphicx}
\usepackage{amssymb,amscd,amsthm,amsxtra}
\usepackage{latexsym}
\usepackage{epsfig}
\usepackage{amsmath}
\usepackage{bm}
\usepackage{amssymb}
\usepackage{amsthm}
\usepackage[usenames,dvipsnames,svgnames,table]{xcolor}
\usepackage[linktocpage=true,colorlinks=true,linkcolor=Blue,citecolor=BrickRed,urlcolor=RoyalBlue]{hyperref}
\usepackage[alphabetic]{amsrefs}
\usepackage{float}

\usepackage[colorinlistoftodos,prependcaption,textsize=tiny]{todonotes}
\usepackage{caption}

\usepackage[mathscr]{euscript} 

\usepackage{mathrsfs} 

\usepackage{pxfonts}

\usepackage[varg]{txfonts}

\numberwithin{equation}{section} 

\newtheorem{theorem}{Theorem}

\newtheorem{lemma}{Lemma}

\theoremstyle{remark}

\def \suchthat {\ \big | \ }

\title[Hessian regularity in the plane]{On the sharp hessian integrability conjecture\\ in the plane}
\author{ Thialita M. Nascimento $\&$ Eduardo V. Teixeira}
\address{Department of Mathematics,  University of Central Florida, Orlando, FL, USA}
\email{thnascimento@knights.ucf.edu}
\email{eduardo.teixeira@ucf.edu}

\begin{document}
\maketitle

\date{} 

\begin{abstract} 
 We prove that if $u\in C^0(B_1)$ satisfies $F(x,D^2u) \le 0$ in $B_1\subset \mathbb{R}^2$, in the viscosity sense, for some fully nonlinear $(\lambda, \Lambda)$-elliptic operator, then $u \in W^{2,\varepsilon}(B_{1/2})$, with appropriate estimates, for a sharp exponent $ \varepsilon = \varepsilon(\lambda, \Lambda)$ verifying
 $$
 	\frac{1.629}{\frac{\Lambda}{\lambda} + 1} < \varepsilon(\lambda, \Lambda)  \le \frac{2}{\frac{\Lambda}{\lambda} + 1},
 $$
uniformly as $\frac{\lambda}{\Lambda} \to 0$. This is closely related to the Armstrong-Silvestre-Smart conjecture, raised in [Comm. Pure Appl. Math. 65 (2012), no. 8, 1169--1184], where the upper bound is postulated to be the optimal one. 
\tableofcontents
\end{abstract}

\section{Introduction}
In the investigation of non-variational diffusion processes, the so-called $W^{2,\varepsilon}$-regularity estimate plays a central role.  Initially noted by F. Lin in \cite{Lin} and later generalized by L. Caffarelli in \cite{Caff} to the fully nonlinear setting, the $W^{2,\varepsilon}$-regularity estimate asserts that the hessian of solutions of uniformly elliptic PDEs is entitled to a {\it universal} integrability control, viz
$$
	\left ( \int_{B_{1/2}} |D^2 u|^\varepsilon dx \right )^{1/\varepsilon} \le C \|u\|_{L^\infty(B_1)},
$$ 
for a positive exponent $\varepsilon>0$ and $C > 1$, both depending only on dimension and ellipticity constants. No a priori continuity assumption is required on the coefficients, and thus this is a foundational result to the development of the theory, which yields higher order regularity for more structured models, see \cite{Caff, CC} for details.  

Original proofs of the $W^{2,\varepsilon}$-regularity estimate give no hint whatsoever on how the sharp hessian integrability exponent $\varepsilon = \varepsilon(n, \lambda, \Lambda)$ depends upon dimension and the ellipticity constants. In particular, no useful information on its magnitude could be obtained. Such a question became even more prominent after the work of Armstrong, Silvestre, and Smart, \cite{ASS}, on partial regularity of solutions to fully nonlinear elliptic equations. Under $C^1$ assumption on the operator $F$, the authors managed to prove that viscosity solutions to
\begin{equation}\label{Intro F=0}
	F(D^2 u) = 0, \quad B_1\subset{\mathbb{R}^n},
\end{equation}
are $C^2$-smooth, outside a set $\mathcal{S}$, of Hausdorff dimension strictly smaller than the ambient dimension $n$. Remarkably, the difference between $n$ and the Hausdorff dimension of the singular set  $\mathcal{S}$ is  controlled by the very exponent $\varepsilon$ from the $W^{2,\varepsilon}$-regularity theory, i.e.
$$
	\dim_{\textrm{H}}(\mathcal{S}) \le n - \varepsilon(n, \lambda, \Lambda).
$$

Singular sets are, in principle, unavoidable, e.g. \cite{NV}, and thus {\it quantitative} estimates on $\varepsilon$ are critical to improve Hausdorff measure controls of their sizes. This is one of the reasons why explicit estimates on the hessian integrability exponent of viscosity supersolutions are paramount to the theory of fully nonlinear PDE. 

In \cite{ASS}, the authors initiated the endeavor of gauging the magnitude of $\varepsilon$ and, by means of a clever construction, they managed to show that the quantity 
\begin{equation}\label{ASS Conj}
    \varepsilon_{\star} = \frac{2}{\frac{\Lambda}{\lambda} + 1}
\end{equation} 
constitutes a universal upper bound for the exponent $\varepsilon$. The example crafted in \cite{ASS} is two-dimensional, and after some heuristics, the authors are led to conjecture that \eqref{ASS Conj} 
is the (universal) optimal exponent in the $W^{2,\varepsilon}$-regularity theory. In higher dimensions, $n\ge 3$, however, it is possible to sharpen the upper bound by 
\begin{equation}\label{NT-CE}
	\bar{\varepsilon}_n:= \frac{n}{(n-1)\frac{\Lambda}{\lambda} + 1},
\end{equation} 
see \cite{Nasc-Teix}. The conjecture remains open in the case $n=2$, and the main theorem proven in this current paper is a step towards understanding the nuances of such an intricate question. We will prove the following:
\begin{theorem} \label{main thm} Let $0<\lambda < \Lambda$ be given and $u \in C(\overline{B}_1)$ satisfy $\mathcal{M}_{\lambda, \Lambda}^{-} (D^2 u) \le 0$ in $B_1 \subset \mathbb{R}^2$. Then $u \in W^{2,\varepsilon} (B_{1/2})$, with appropriate estimates, for an optimal exponent $\varepsilon(\lambda, \Lambda)>0$ that satisfies:
$$
	 \frac{1.629}{\frac{\Lambda}{\lambda} + 1} < \varepsilon(\lambda, \Lambda) \le \frac{2}{\frac{\Lambda}{\lambda} + 1}.
$$
\end{theorem}
Naturally, if $\frac{\lambda}{\Lambda}$ is close to 1, then $\varepsilon(\lambda, \Lambda) $ is too asymptotically close to 1, and the analysis carried out in this paper yields explicit (quantified) lower bounds for $\varepsilon(\lambda, \Lambda)$, as $\frac{\lambda}{\Lambda}$ varies within the open interval $(0,1)$.  The main information given by Theorem \ref{main thm} is that the sharp hessian integrability exponent in the plane remains {\it at least} 81.45\% of the upper bound, uniformly as $\frac{\lambda}{\Lambda}$ tends to zero. 

It is also interesting to note that, as a consequence of Theorem \ref{main thm}, one has $\varepsilon(\lambda, \Lambda) > \frac{\lambda}{\Lambda}$, provided $0< \frac{\lambda}{\Lambda} < 0.629$. Hence, in view of the upper estimate $\varepsilon(\lambda, \Lambda, n) < \bar{\varepsilon}_n$, mentioned above, one deduces that it is indeed impossible to obtain an adimensional Hessian integrability exponent for the $W^{2,\varepsilon}$-regularity theory.

We conclude this introduction by discussing the main ingredients needed in proof of Theorem \ref{main thm}. We use a refinement of the so-called {\it method of sliding paraboloids}, originally introduced by X. Cabre in \cite{Cabre} and O. Savin in \cite{Savin}; see also \cite{Le, Mooney} for further applications of the method to Hessian integrability estimates. Following ideas introduced in \cite{Nasc-Teix}, we delegate the choice of the best dyadic openings to an optimization procedure; a sort of {\it intrinsic scaling} of the problem.  The method then generates a family of extremum problems involving a parameter $c = c(\lambda, \Lambda)$. The main new tool developed in this paper is an essentially sharp measure decay estimate in the plane, which allows one to extract the best possible constant $c$ for the purpose of estimating $\varepsilon(\lambda, \Lambda) $ from below. A linearization analysis then yields improved lower bounds for the Hessian integrability exponent $\varepsilon(\lambda, \Lambda)$, as stated in Theorem \ref{main thm}.


\section{Preliminaries} \label{sct Prelim}
In this preliminary section, we give some standard definitions that will be used throughout the paper. We start off with the notion of Pucci's extremal operators. For a positive integer $n \ge 2$,  let $\mathcal{S}(n)$ be the set of all real $n \times n$ symmetric matrices.  We define for $0 < \lambda \le \Lambda$ and $M \in \mathcal{S}(n)$, the {Pucci's extremal operators} by
\begin{equation}\label{Pucci def}
	\mathcal{M}_{\lambda,\Lambda}^{-} (M) = \inf\limits_{\lambda \text{Id}_n \le A \le \Lambda \text{Id}_n} \textrm{Trace} (AM) \quad \mbox{and}\quad \mathcal{M}_{\lambda, \Lambda}^{+} (M) =  \sup\limits_{\lambda \text{Id}_n \le A \le \Lambda \text{Id}_n} \textrm{Trace} (AM)
	\end{equation}
where the $\inf$ and $\sup$ are taken over all symmetric matrices $A$ whose eigenvalues belong to $[\lambda, \Lambda].$ 

In this article we are interested in viscosity supersolutions of {\it any} possible $(\lambda, \Lambda)$-uniformly elliptic operator $F(x, D^2u)$. This is equivalent to investigating the extremal equation
\begin{equation}\label{main eq}
    \mathcal{M}_{\lambda, \Lambda}^{-} (D^2 u) \le 0,
\end{equation}
in the viscosity sense. For the notion of viscosity solutions, please visit the book \cite{CC} or the classical paper \cite{CIL}.

As usual, we will establish a hessian integrability result by measuring the decay of the minimum curvature of paraboloids touching $u$ from below at a given point $x_0 \in\Omega$. This leads to the definition of the function $\underline{\Theta} (u,\Omega) (x_0) = \underline{\Theta}  (x_0)$ given as
\begin{equation}\label{eq2.2}
	\underline{\Theta}  (x_0):=  \inf \left \{  A >0 \suchthat u(x) \ge u(x_0) + y\cdot (x-x_0) - \frac{A}{2}|x-x_0|^2 \text{ in } \Omega, \text{ for some } y \in \mathbb{R}^n \right \}.  
\end{equation}
If there is no tangent paraboloid from bellow at $x_0$,  we set $\underline{\Theta} (u,\Omega) (x_0) = + \infty$.  For $a > 0$ and $L(x)$ an affine function, we say  
$$
	P_{L}^{a} (x) = - \frac{a}{2} |x|^2 + L(x)
$$
is  a paraboloid of opening $-a$. If $\Omega$ is a bounded, strictly convex domain and $v \in C( \overline{\Omega})$, we define the $a$-convex envelope $\Gamma_{v}^a$ on $\overline{\Omega}$ as 
$$
	\Gamma_{v}^a (x) := \sup\limits_{L} \left \{P_{L} ^{a} (x) : P_{L} ^{a} \le v \,\, \mbox{in}\,\, \overline{\Omega} \right \}.
$$
By convexity of $\Omega$, one verifies that $ \Gamma_{v}^a = v$ on $\partial \Omega$. Also, taking $a = 0$, one recovers the usual notion of convex envelopes by affine functions.  Next we  define 
$$
	A_a (v) : = \left \{ x \in \Omega \suchthat v(x) = \Gamma_{v}^a (x) \right \}.
$$
That is, $A_a(v) $ is the set of points in $\Omega$ where $v$ has a tangent paraboloid of opening $-a$ from below in $\Omega.$ One can check that $A_a(v)$ is closed in $\Omega$ and that $A_a(v) \subset A_b (v)$ if  $a \le b$ . Also, it is easy to verify that for any $\lambda, \, \gamma \in \mathbb{R}$ and $\beta > 0$, there holds:
\begin{equation}\label{eq2.3}
	\Gamma_{\beta v + \frac{\gamma}{2} |x|^2}^{\lambda} = \beta \Gamma_{v}^{\frac{\lambda + \gamma}{\beta}}  + \frac{\gamma}{2} |x|^2   \quad \mbox{ and } \quad  A_{\lambda} \left(\beta v + \frac{\gamma}{2} |x|^2\right) = A_{\frac{\lambda + \gamma}{\beta}} (v).
\end{equation}

\section{A new measure estimate in the plane} \label{sct new lemma}

In this section, we obtain an improved estimate on the measure decay of the contact sets of a continuous function $u$ of two variables, satisfying \eqref{main eq}, and its lower envelope of paraboloids. Such a tool will foster an improved parameter $c(\lambda, \Lambda)$ for the extremum problem used to bound from below the hessian integrability exponent in the $W^{2,\varepsilon}$-regularity theory. 

\begin{lemma}\label{measure lemma}
	Let $u \in C (\overline{\Omega})$ and  assume $ \mathcal{M}_{\lambda,\Lambda}^{-} ( D^2 u) \le 0$ in $\Omega \subset \mathbb{R}^2$.  Given $a, \delta > 0$ and a measurable set $F \subset\{u > \Gamma_u ^a\}$, take the paraboloids of opening $-(1 + \delta)a $, tangents from below to $\Gamma_u ^a$ on $F$, and slide them up until they touch $u$ on a set $E $. If $E \Subset \Omega$ then, 
	\begin{equation}\label{Thesis l3}
		| A_{(1+ \delta)a} (u) \setminus A_a(u) | \ge c  \left(1 + \frac{1}{\delta}\right)^{-2} |F|,
	\end{equation}
	where $c = c( \lambda, \Lambda)$ is given by:
	\begin{equation}\label{c in l3}
		c  :=  \left[ 1 + \frac{1}{4} \frac{\Lambda}{\lambda} \left( 1 - \frac{\lambda}{\Lambda}\right)^2 \right] ^{-1}.
	\end{equation}
\end{lemma}
\begin{proof} 
In view of \cite[Lemma 3.1 and Lemma 3.2]{Mooney} and the subsequent analysis carried out in \cite{Nasc-Teix}, for $\delta, a > 0$, to be chosen {\it a posteriori}, the new contact point $E \subset \overline{\Omega}$ satisfies:
	$$
		E \subset A_{\delta} (v) \setminus A_0 (v) =  A_{(1+ \delta)a} (u) \setminus A_a(u),
	$$	
where $\displaystyle v= \frac{1}{a}u + \frac{1}{2}|x|^2 $. Also, if $V_F$ denotes the set of vertices of all tangent paraboloids of opening $-\delta$, touching $u$ from below at a point $x_{v}^0$ in $F$, then
	\begin{equation}\label{l3-eq.1}
	|V_F| \ge |F|.
	\end{equation}
	The vertexes of such paraboloids remain the same, after we slide them up, and thus 
	$$
		V_{F} = V_E.
	$$
	For the time being, we assume that $u$ is semi-concave, and thus twice differentiable a.e in $E$. For such points $x \in E$,  the correspondent vertex is given by 
	$$ 
		\Phi(x) := x_V = x + \frac{1}{\delta} \nabla v(x).
	$$
	Since $\Phi (x)$ is a Lipschitz map, we can apply the area formula and reach
	\begin{equation}\label{l3-eq.2}
		|V_E| = |\Phi(E)| \le \int_{E} |\det (D \Phi) |  dx.
	\end{equation}
	By the definition of $\Phi$ (and of $v$), the eigenvalues of $D \Phi$ are of the form 
	$$ 
		\lambda_i ^{\Phi} = \left(1 + \frac{1}{\delta}\right) + \frac{1}{\delta a} \lambda_i ^{u},\quad  i = 1, 2,
	$$
	where $\lambda_i^{u}$ are the eigenvalues of $D^2 u$. Now, because $ x \in E \subset A_{(1+\delta)a}(u) \setminus A_{a}(u)$, we have 
	$$ 
		D^2 u(x) \ge -(1+\delta)a \cdot \text{Id}.
	$$ 
	Thus, the eigenvalues of $D\Phi (x)$ are all non-negative. Next, since $\mathcal{M}_{\lambda, \Lambda}^{-} (D^2 u) \le 0$,   the Hessian of $u$ has at least one non-positive eigenvalue, say $\lambda_1^{u} \le 0.$ If both eigenvalues are non-positive, $u$ is concave and we can estimate
	$$
	    	\det (D \Phi) = \lambda_1^{\phi} \cdot \lambda_2^{\phi} \le \left( 1 + \frac{1}{\delta} \right)^2.
	$$
	If $\lambda_1^{u} \le 0 \le \lambda_2^{u}$, the equation $\mathcal{M}_{\lambda, \Lambda}^{-}(D^2 u (x)) \le 0$, yields:	
	\begin{eqnarray}
		\lambda_1^{u} \le - \frac{\lambda}{\Lambda} \lambda_2^{u}. \nonumber
	\end{eqnarray}
	In this case, the determinant of $D \Phi$,
	\begin{eqnarray}\label{l3-eq.3}
			\det (D \Phi) &=& \lambda_1^{\phi} \cdot \lambda_2^{\phi} \nonumber \\
			&= & \left[ \left( 1 +\frac{1}{\delta} \right) + \frac{1}{\delta a} \lambda_1^u \right] \left[\left( 1 +\frac{1}{\delta} \right) + \frac{1}{\delta a} \lambda_2^u \right], \nonumber
\end{eqnarray}
can be seen as a function of $(\lambda_1^{u} , \lambda_2^{u})$ defined on the triangle
$$
	T = \left[-(1 + \delta)a, 0\right] \times \left[0, -\frac{\Lambda}{\lambda} \lambda_1^{u}\right],
$$ 
obtained by varying $\lambda_1^{u}$ within $\left[-(1 + \delta)a, 0\right]$. Easily one checks that the maximum of such a function is attained along the line $\displaystyle \lambda_{2}^{u} = -\frac{\Lambda}{\lambda} \lambda_1^{u}$.  Therefore,  we can estimate
\begin{equation}\label{l3-eq.3}
	\begin{array}{lll}
			\displaystyle \det (D \Phi)	  &\le& \displaystyle \left[ \left( 1 +\frac{1}{\delta} \right) - \frac{\lambda}{\Lambda\delta a} \lambda_2^u \right] \left[ \left( 1 +\frac{1}{\delta} \right) + \frac{1}{\delta a} \lambda_2^u \right]\\
			&=& \displaystyle \left(1+ \frac{1}{\delta} \right)^2 +\frac{1}{\delta a}\left(1+ \frac{1}{\delta} \right) \left( 1 - \frac{\lambda}{\Lambda}\right)\lambda_2^u - \frac{1}{(\delta a)^2}\frac{\lambda}{\Lambda} (\lambda_2^u)^2\\
			&\le&  \displaystyle  \max\limits_{\lambda_2^{u}} \left(\left(1+ \frac{1}{\delta} \right)^2 +\frac{1}{\delta a}\left(1+ \frac{1}{\delta} \right) \left( 1 - \frac{\lambda}{\Lambda}\right)\lambda_2^u - \frac{1}{(\delta a)^2}\frac{\lambda}{\Lambda} (\lambda_2^u)^2  \right) \\
			&=& \displaystyle \left(1+ \frac{1}{\delta} \right)^2 \left[ 1 + \frac{1}{4} \frac{\Lambda}{\lambda} \left( 1 - \frac{\lambda}{\Lambda}\right)^2 \right].
	\end{array}
\end{equation}
 
 By letting $\displaystyle c^{-1} = \left[ 1 + \frac{1}{4} \frac{\Lambda}{\lambda} \left( 1 - \frac{\lambda}{\Lambda}\right)^2 \right] $, in both cases we can  estimate:
	\begin{equation}\label{key est}
		\begin{array}{lll}
			\displaystyle |F|\le |V_E| \displaystyle	&\le& \int_{E} \det(D \Phi (x) )  \\
					&\le&  \displaystyle c^{-1}\left(1 + \frac{1}{\delta}\right)^2  |E|   \\
					&\le & \displaystyle c^{-1}\left(1 + \frac{1}{\delta}\right)^2  \left | A_{\delta} (v) \setminus A_0 (v) \right |   \\
					& =& \displaystyle c^{-1}\left(1 + \frac{1}{\delta}\right)^2  \left | A_{(1 + \delta)a}( u) \setminus A_a (u) \right |,
		\end{array}
	\end{equation}
	and the lemma is proved for semi-concave functions. For the general case, we reduce it to the previous one by using the inf-convolution. Define
	$$
		u_{m} := \inf\limits_{y \in \Omega} \left \{ u(y) + m | y -x |^2 \right \}.
	$$	
	One can show, see e.g. the proof of \cite[Lemma 2.1]{Savin},  that $u_{m}$ are semi-concave and converge locally uniformly to $u$ in $\Omega$ as $m \to \infty$. Furthermore $u_{m}$ is a viscosity supersolution of the same equation in $\Omega^{\prime} \Subset \Omega$ and by the first part of the proof, 
	$$
		|E_{m}| \ge \beta |F| 
	$$
	with $\displaystyle \beta =  c  \left(1 + \frac{1}{\delta}\right)^{-2}$ and where $E_{m} $ is the correspondent touching set for $u_{m}$. One can easily check that
	$$
		\limsup\limits_{m \to \infty} E_{m}  = \bigcap\limits_{m \ge 1} \bigcup\limits_{i \ge m} E_{i} \subset E,
	$$
	and thus, 
	$$
		|A_{(1 + \delta)a}( u) \setminus A_a (u) | \ge \beta |F|.
	$$
	The proof of the Lemma is complete.
\end{proof}

\section{Scaling optimization} \label{NT1}

In this section, we  give a brief description of the strategy put forward in \cite{Nasc-Teix}, leading to the critical function $\varepsilon \colon (0,1)\times (0,1] \to (0,1]$,
\begin{equation}\label{critical function eps}
    \varepsilon(x, c) :=\frac{\ln(1- cx^n)}{\ln(1- x)},
\end{equation}
where $n\ge 2$ is the ambient dimension. 

Classically, lower bounds for an optimal integrability exponent $\varepsilon_{\star}$ in the $W^{2,\varepsilon}$-regularity theory comes from the decay estimate 
\begin{equation}
    \left| \left\lbrace \underline{\Theta}_{u} >  t \right\rbrace  \cap B_{1/2} \right|  \le C t^{- \varepsilon},
  \nonumber
\end{equation}
which implies that $\varepsilon \le \varepsilon_{\star}$, see for instance \cites{Caff, CC}. The method relies on a dyadic decomposition for the level set of $\underline{\Theta}_{u}$.  Since 
$$
	\{x\in B_1 \suchthat  \underline{\Theta}_{u} > t \} \subset B_1 \setminus A_{t} (u)
$$ 
it suffices to show decay in measure of the sets $B_1 \setminus A_{t} (u)$. For any $t \ge 2$, one has 
$$ 
    B_1 \setminus A_{2^{k+1}}(u) \subset A_{t} (u) \subset B_1 \setminus A_{2^k}(u),
$$
 for some $k \in \mathbb{N}$, and thus the idea is to estimate the decay in measure of the sets $B_1 \setminus A_{2^k}(u)$. Mooney, in \cite{Mooney}, implements the method of sliding paraboloids, avoiding the localization argument used by Caffarelli in \cites{Caff, CC}. This leads to an estimate of the form
\begin{equation}\label{Mooney's measure decay}
    |A_{2^{k+1}}(u) \setminus A_{2^k} (u) | \ge c_0 |B_1 \setminus A_{2^{k}}(u)|
\end{equation}
where $c_0$ depends only on the ellipticity constants and dimension.  

In \cite{Nasc-Teix}, the authors introduced an intrinsic optimization process within the recursive measure decay algorithm. If the dyadic radii for the level sets of $\underline{\Theta}_{u}$ is taken to be (for now) an arbitrary number $r=1 + \delta, \, \delta > 0$, a similar analysis yields an estimate analogous to \eqref{Mooney's measure decay}; however the corresponding constant $c_0$ now depends on the ellipticity constants, dimension, as well as on $\delta$. A change of variables, $x := (1+ \delta^{-1})^{-1}$, then leads to critical function \eqref{critical function eps}, where $c = c(\lambda, \Lambda, n)$ is the very same constant appearing in the corresponding estimate \eqref{Thesis l3}. Hence, in view of Lemma \ref{measure lemma}, in the plane we are entitled to use the constant in \eqref{c in l3}.

For each $x\in (0,1)$, the value $\varepsilon(x, c) $ can be used as to estimate from below the hessian integrability exponent of viscosity supersolution. Thus, the idea now is to optimize, {\it a posteriori}, the choice of the dyadic radii, by calculating the supremum of $\varepsilon(x, c)$ as $x$ runs from $0^{+}$ to $1^{-}$. If $x_\star$ is the critical point, then $r_\star = 1+ \delta_\star$, where $x_\star = (1+ \delta_\star^{-1})^{-1}$, is the best (intrinsic) dyadic radius for the recursive algorithm. In conclusion, one can estimate $\varepsilon$ from below by  $\varepsilon_{\star} = \sup\limits_{(0,1)} \frac{\ln(1- cx^n)}{\ln(1- x)}$.

\section{Linearization analysis} \label{sct lower estimates}

In view of Lemma \ref{measure lemma}, and following the strategy explained in the previous Section, we are led to the analysis of the function
\begin{equation}\label{th1-ep}
	 	\varepsilon(\lambda, \Lambda) := \sup\limits_{(0,1)}\frac{\ln(1- c(\lambda, \Lambda)x^2)}{\ln(1- x)},
\end{equation} 
	where, 
\begin{equation}\label{th1-ep2}
		c(\lambda, \Lambda)= \left[ 1 + \frac{1}{4} \frac{\Lambda}{\lambda} \left( 1 - \frac{\lambda}{\Lambda}\right)^2 \right] ^{-1} = \frac{4\Lambda}{\lambda (1 + \frac{\Lambda}{\lambda})^2}.
\end{equation}
For any $0< c< 1$, one can estimate  
$$
	-\ln(1- cx^2) = - c \ln\left ((1- cx^2)^{1/c} \right ) \ge -c x^2,
$$
with asymptotic equality as $0< c\ll 1$.  Thus, we can write down the pointwise inequality  
\begin{equation} \label{assintotic ineq for eps}
    	\varepsilon(x, \lambda, \Lambda)  := \frac{\ln(1- c(\lambda, \Lambda)x^2)}{\ln(1- x)} > \frac{c(\lambda, \Lambda)x^2}{- \ln(1-x)}, \quad x \in (0,1).
\end{equation}
In particular, one has $\varepsilon(\lambda, \Lambda)  \ge c(\lambda, \Lambda) m_0(2)$, where 
\begin{equation}\label{m0}
    m_0(2) = \sup_{(0,1)}\left ( \frac{x^2}{- \ln(1-x)} \right ) > 0.4073.
\end{equation}
Hence, 
\begin{equation}\label{lim}
	\varepsilon(\lambda, \Lambda) > \frac{4\Lambda}{\lambda (1 + \frac{\Lambda}{\lambda})^2} m_0(2) > \left (1-\text{o(1)} \right ) \frac{1.629}{\frac{\Lambda}{\lambda} +1},
\end{equation}	
as $ \frac{\lambda}{\Lambda} \to 0$. This is exactly the asymptotic thesis of Theorem \ref{main thm}. 

To obtain refined estimates, let us now return to the optimization problem \eqref{th1-ep}. For $0< c< 1$ fixed, consider the parabola
$$
	p_c(t) = 1- ct^2,
$$
defined over the open interval $(0,1)$. Now, for $d \in (0,1]$, let us inspect the family of curves $\varphi_d\colon (0,1) \to (0,1)$ given by
$$
	\varphi_d(t) = (1-t)^d.
$$
When $d= 1$, the corresponding function $\varphi_1(t) = (1-t)$ is strictly below $p_c(t)$. As we decrease $d$ continuously, the graph of $\varphi_d(t)$ bends towards  the graph of $p_c(t)$. Let $0<d_c<1$ be the first value such that  $\varphi_d(t)$ touches $p_c(t)$, i.e.
$$
	d_c := \inf \left \{ d \in (0,1] \suchthat \varphi_d(t) < p_c(t), ~ \forall t\in (0,1) \right \}.
$$
Let $x_c$ be the point in $(0,1)$ determined by the equation $\varphi_{d_c}(x_c) = p_c(x_c)$. The following system of equations then holds:
\begin{equation}\label{sys}
	\left \{
		\begin{array}{cll}
			\displaystyle 1-cx_c^2 &=& \displaystyle (1-x_c)^{d_c} \\
			\displaystyle 2cx_c &=& \displaystyle d_c(1-x_c)^{d_c-1}.
		\end{array}
	\right.
\end{equation}

Easily one can prove the following:

\begin{lemma}\label{l1} Let $(x_c, d_c)$ be the only solution of the system \eqref{sys}. Then 
$$
	d_c =  \sup\limits_{(0,1)}\frac{\ln(1- c x^2)}{\ln(1- x)} =   \frac{\ln(1- cx_c^2)}{\ln(1- x_c)}.
$$
\end{lemma}

Also, one easily checks that
$$
	x_c = \dfrac{1 + \sqrt{1 - \frac{d_c}{c} (2-d_c)}}{2-d_c},
$$
and that the function $\psi (c) = \frac{d_c}{c}$ is increasing in $c$, with $\lim\limits_{c\to 0} \psi(c) = m_0(2).$ We define  
\begin{equation}\label{x0}
	x_0 :=  \dfrac{1 + \sqrt{1 - 2m_0(2)}}{2} \approx 0.715.
\end{equation}
Finally, in view of Lemma \ref{l1}, \eqref{th1-ep} and \eqref{th1-ep2}, let us parametrize the analysis above by the ellipticity ratio $\tau := \frac{\lambda}{\Lambda} \in (0,1)$, that is, define:
$$
	c(\tau) = \frac{4}{\tau (1+\tau^{-1})^2}, \quad \varepsilon (\tau) = \sup\limits_{(0,1)}\frac{\ln(1- c(\tau)x^2)}{\ln(1- x)}, \quad x(\tau) = \dfrac{1 + \sqrt{1 - \frac{ \varepsilon (\tau) }{c(\tau)} (2- \varepsilon (\tau) )}}{2- \varepsilon (\tau) }.
$$
As before, it is easy to see that the function $\tilde{\psi}(\tau) := \left ( \tau^{-1} + 1 \right ) \varepsilon(\tau) $ is too increasing with respect to $\tau$. In particular, according to \eqref{lim},
$$
	\left ( \tau^{-1} + 1 \right )\varepsilon(\tau) \ge \lim\limits_{\tau \to 0}  \varepsilon(\tau) > 1.629,
$$ 
for all $\tau \in (0,1)$.

To obtain even better lower bonds for $\varepsilon$, one can evaluate the function $\varepsilon(x, \tau)$ defined \eqref{assintotic ineq for eps} at a point obtained by the interpolation between $x_0$, obtained in \eqref{x0}, and $x_1 = 1$. Namely, if one defines $t(\tau) := x_0 + (1-x_0) c(\tau)^{2.4},$	
the following estimate 
$$
	\varepsilon(\tau) \ge \frac{\ln \left ( 1- c(\tau) t^2(\tau) \right )}{\ln \left ( 1-  t(\tau) \right )},
$$
holds for all $0< \tau < 1$.

\begin{center}
		\includegraphics[scale=0.29]{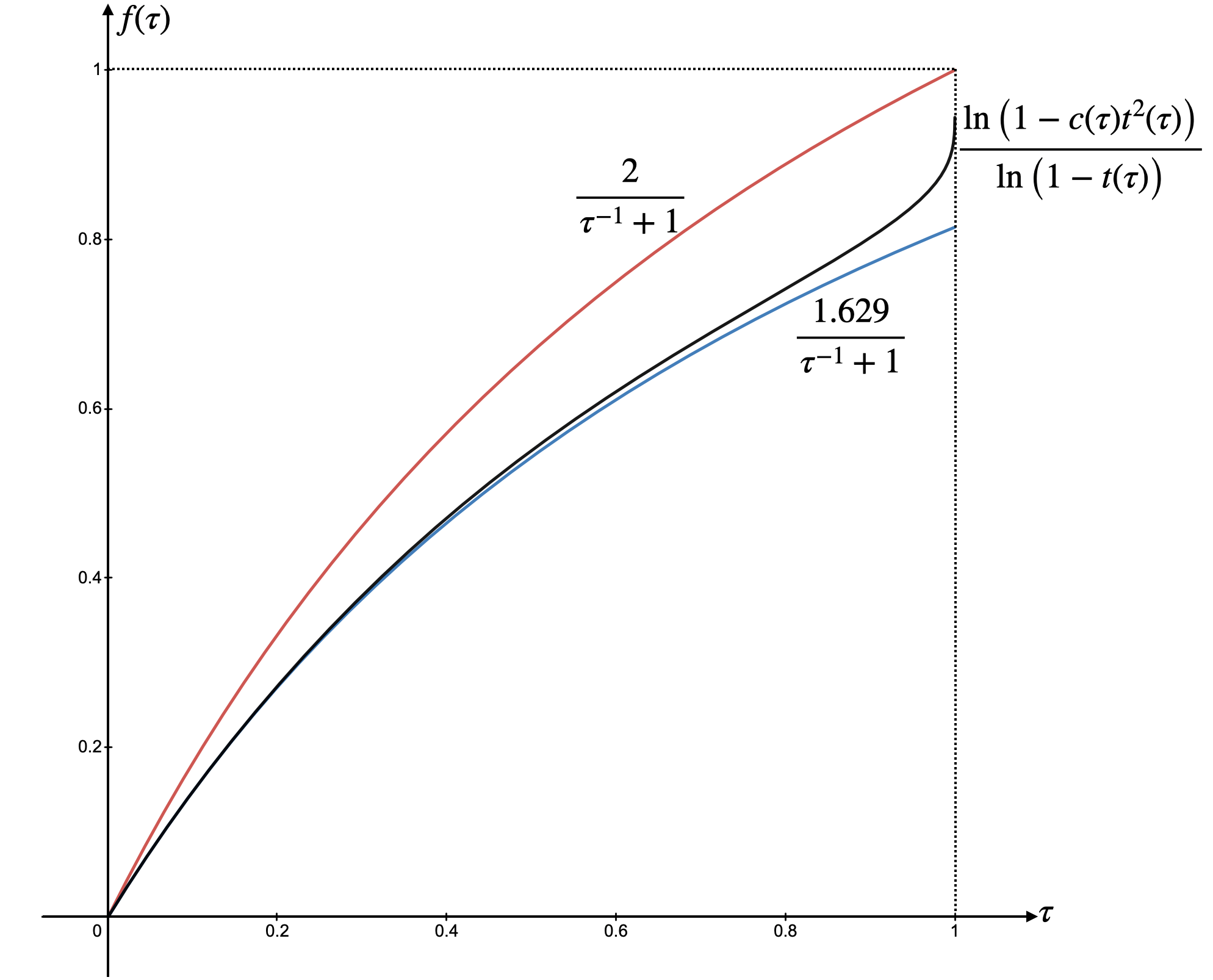}
\end{center}

For a fixed ellipticity ratio $\tau = \frac{\lambda}{\Lambda}$, better estimates can be obtained numerically by approximating the solution of the optimization problem \eqref{th1-ep}.

\bibliographystyle{amsplain, amsalpha}

\end{document}